\documentclass{article}
\usepackage{amsmath}
\usepackage{amsfonts}
\usepackage{amssymb}
\usepackage{graphicx}
\usepackage{xcolor}

\usepackage[
  left=2.5cm,
  right=2.5cm,
  top=3cm,
  bottom=3cm
]{geometry}
\usepackage{hyperref}
\hypersetup{
    colorlinks=true, 
    linktoc=all,     
    linkcolor=blue,  
    urlcolor=blue,
}

\usepackage{color} 
\usepackage[normalem]{ulem} 
\usepackage{todonotes}      

\usepackage[nottoc]{tocbibind}

\newtheorem{theorem}{Theorem}[section]

\newtheorem{example}[theorem]{Example}

\newtheorem{lemma}[theorem]{Lemma}

\newenvironment{proof}[1][Proof]{\noindent\textbf{#1.} }{\ \rule{0.5em}{0.5em}}

\numberwithin{equation}{section}

\newcommand{\dd}{\mathrm{d}}

\def\Xint#1{\mathchoice
   {\XXint\displaystyle\textstyle{#1}}%
   {\XXint\textstyle\scriptstyle{#1}}%
   {\XXint\scriptstyle\scriptscriptstyle{#1}}%
   {\XXint\scriptscriptstyle\scriptscriptstyle{#1}}%
   \!\int}
\def\XXint#1#2#3{{\setbox0=\hbox{$#1{#2#3}{\int}$}
     \vcenter{\hbox{$#2#3$}}\kern-.5\wd0}}

\def\dashint{\Xint-}

\begin{document}

\title{Local Boundedness of Local Minimizers for a Class of Nonlinear Elliptic Systems with General Growth}
\author{Elvira Mascolo$^{1}$, Antonella Nastasi$^{2}$, Cintia Pacchiano
Camacho$^{3}$ \\
\medskip \\
{\normalsize $^{1}$Dipartimento di Matematica e Informatica
\textquotedblleft U. Dini\textquotedblright , Universit\`{a} di Firenze}\\
{\normalsize Viale Morgagni 67/A, 50134 - Firenze, Italy}\\
{\normalsize elvira.mascolo@unifi.it}\\
{\normalsize $^{2}$Department of Engineering, University of Palermo}\\
{\normalsize Viale delle Scienze, 90128 - Palermo, Italy}\\
{\normalsize antonella.nastasi@unipa.it}\\
{\normalsize $^{3}$ Instituto de Matem\'aticas, Unidad Cuernavaca}\\ {\normalsize Universidad Nacional Aut\'onoma de M\'exico}\\
{\normalsize Av. Universidad, 62210 - Cuernavaca, Morelos, Mexico}\\
{\normalsize cintia.pacchiano@im.unam.mx }}
\date{}
\maketitle

\begin{abstract}
In this paper, we prove the local boundedness of solutions to systems of partial differential equations in divergence form. More specifically, we consider systems that include the first variations of functionals depending on the spatial variable and exhibiting nonstandard growth with respect to the gradient, such as 
\[
\int_{\Omega} \left( 1+ h(|Du|)\right)
^{\alpha(x)} \,\dd x,
\]
where the convex function $h=h(t)$ does not satisfy the so-called 
$\Delta_2$ property and does not exhibit the conventional polynomial growth behavior.
\end{abstract}

 

\tableofcontents

\section{Introduction}

Let $\Omega$ be an open set in $\mathbb{R}^n$, $n \ge 2$.  
Let $F = F(x,\xi)\geq 0$, $F : \mathbb{R} \times \mathbb{R}^{nN} \to \mathbb{R}$ be a convex function and 
\[
|\xi| := \left( \sum_{\alpha=1}^N \sum_{i=1}^n |\xi_{i}^{\alpha}|^2 \right)^{1/2}.
\]
We consider the functional
\begin{equation}\label{eq1.1}
\mathcal{F}(u) := \int_{\Omega} F(x,Du)\,\dd x,
\end{equation}
where $Du\in \mathbb{R}^{nN}$ denotes the gradient of a vector-valued function
$u : \Omega\subset\mathbb{R}^n\to \mathbb{R}^N$.\\

In the vector-valued case, as suggested by well known counterexamples by De Giorgi \cite{deg}, generally  a structure  depending on $|Du|$ on the integrand, are required for everywhere regularity. Then we assume that 
\begin{equation} \label{funzionale}
F(x,\xi) = g(x,|\xi|), \qquad \xi \in \mathbb{R}^{nN},
\end{equation}
 where
 \[
|Du(x)|
=
\left( \sum_{i=1}^N |\nabla u_i(x)|^2 \right)^{1/2}
=
\left( \sum_{i=1}^N \sum_{\alpha=1}^n 
\left| \frac{\partial u_i(x)}{\partial x_\alpha} \right|^2
\right)^{1/2}.
\]

Given a local minimizer $u$  it is relevant to  study the  properties of the energy function in order to establish the regularity of $u$, in the
interior of the domain.\\

In particular we investigate the following question: 
suppose that a local minimizer $u$ of integral functional $\mathcal{F}(u)$ in \eqref{funzionale} exists. Is such a minimizer bounded?

\vskip.3cm

The function $g: \Omega\times [0,+\infty) \to [0, +\infty)$ is a 
generalised Young function  in $\Omega$:
\begin{itemize}
	\item  [(a)] $g(\cdot,t)$ is a measurable function on $\Omega$, for 
	all $t\geq 0$;

\item [(b)] for almost every $x \in \Omega$, $g(x,\cdot)$ is a convex 
function on $[0,+\infty)$ such that $g(x,t)=0$ if and only if $t=0$, 
and
$$
\lim_{t\to +\infty}\frac{g(x,t)}{t}=+\infty\,;
$$
\item [(c)] for a.e. $x \in \Omega$, $g(x, \cdot)$ is differentiable in $[0,+\infty)$, and its derivative $g_t(x, t)$ is a
Caratheodory function in $\Omega \times [0,+\infty)$.
\end{itemize}

\medskip

In the case where $g(x,t)=g(t)$, then the Definition above 
corresponds to the usual definition of  Young function \cite{RaoRen, KrRu}.

\vskip.3cm

Let us define
$$
V= \left\lbrace  u \in W^{1,1}_{\rm{loc}}(\Omega,\mathbb{R}^N): \int_{\Omega} g(x,|Du|)\,\dd x < +\infty \right\rbrace 
$$

In the following we denote  \textbf{local minimizer} of \eqref{eq1.1} a  Sobolev function $u \in W^{1,1}_{\rm{loc}}(\Omega, \mathbb{R}^N)$ such that
\[
\mathcal{F}(u, \operatorname{supp}\varphi)
:=
\int_{\operatorname{supp}\varphi} g(x,|Du|)\,\dd x
\le
\int_{\operatorname{supp}\varphi} g(x,|Du + D\varphi|)\,\dd x
\]
for all $\varphi \in V$ with
$\operatorname{supp}\varphi \Subset \Omega$ and such that
$u+\varphi \in V$.\\

In this paper we consider a  general framework and
we establish the local boundedness of vectorial local minimizers for a 
specific class of integral functionals under appropriate structural assumptions on the integrand function $g=g(x,t)$  and on its derivatives with respect to $x$.
\vskip.3cm
We assume that 
\begin{equation}\label{ipotesi1}
\begin{cases} 
 &\mbox{for every $\lambda>1$ there exists $\sigma,\,p>1$ such that }\\& g(x,\lambda t) \le C \,\lambda^p g^\sigma(x,t),\,\, \mbox{for a.e.  $x\in\Omega$, for every $t \geq 0$,}
\end{cases}
\end{equation}
and
\begin{equation}\label{H_3'}
\begin{cases}&\mbox{for every $t \geq 0$ and for every $i = 1,..., n$, $g(\cdot, t)$ has weak derivatives $g_{x_i}(\cdot, t)$
 in $L^1_{\rm{loc}}(\Omega)$,}\\  &\mbox{which are Caratheodory functions in
 $\Omega \times [0,+\infty)$ and there exists a function $\gamma(x) \in L^s(\Omega)$ }\\ 
&\mbox{ with $s>pn$ and $\nu < \frac{s-n}{s(n-1)}<1$ such that}\\ 
& \hspace{3cm}|g_x(x,t)| \le C\, \gamma(x)  \, g(x,t)(1+g^{\nu}(x,t)),\,\, \mbox{for a.e.  $x\in\Omega$, for every $t \geq 0$,} \\ 
&\mbox{where $g_x$ denotes the weak gradient of $g$ with respect to $x$.}
\end{cases}\end{equation}

Of course it is not a
restriction to assume that 
$\gamma(x) \geq 1$ almost everywhere.

\vskip.3cm

A main novelty with respect to the existing literature on this subject is that we assume that g satisfies \eqref{ipotesi1}; consequently, the energy density is not of $\Delta_2$-class with respect to $t$. \\

We recall that a Young function $g$   satisfies the  $\Delta_2$ condition with respect to $t$ when for $\lambda>1$
$$g(x,\lambda t) \le C \,\lambda^m g (x,t), \quad m>1$$
 i.e. \eqref{ipotesi1} with $\sigma=1$, which implies 
 $$g(x,t) \le C \,t^m, \quad t\geq 1.$$
\vskip.3cm
For the theory of Young functions and the related Orlicz and Orlicz–Sobolev spaces, we refer to \cite{RaoRen} and \cite{Har-Hasto}. These books do not focus directly on  the regularity, but provide the functional-analytic framework needed for boundedness proofs.\\

We observe that when $g \in \Delta_2$, the Orlicz space behaves almost like $L^p$, therefore the Moser iteration procedure then proceeds in a way similar to the power-growth case.
For $g$  that does not satisfy the $\Delta_2$ condition, from the algebrical point of view  one cannot control
$g(x,t_1+t_2)$, and from functional framework the set of $u$ such that $g(x,u) \in L^1$ is not a vector space.
Moreover, in the absence of the $\Delta_2$-condition, even the derivation of the Euler equation requires additional assumptions, making the boundedness theory significantly more delicate.
\vskip.3cm

In this paper we aim to provide  a first answer concerning the boundedness regularity of minimizers without assuming the $\Delta_2$ property on the energy function.\\

We present the main result of the manuscript, which is a local estimate of the supremum
of $|u|$ in terms of the integral which defines functional $\mathcal{F}$.
\begin{theorem}\label{boundedness theorem}
Let $g$ satisfy hypotheses \eqref{ipotesi1} and \eqref{H_3'} with 
\begin{equation}\label{boundsigma}
 \sigma < 1^*\,\frac{s-p}{s}< 1^*:=\frac{n}{n-1}.
\end{equation}
Let $u$ be a local minimizer of $\mathcal{F}$ in
$V$ such that there exists $\tau>1$ with
\begin{equation}\label{gTau}
\int_{\Omega} g(x,\tau |Du|)\,\dd x < \infty .
\end{equation}
Then $u$ is locally bounded. Moreover, for all $\Omega_0$ open set compactly contained in $\Omega$, $x_0\in \Omega_0$, for every $\rho$, $R$ such that $0<\rho<R\leq \min\{\rm{dist}(x_0, \partial\Omega_0),1\}$ and for every $\alpha>1$, the following inequality holds
\begin{equation}\label{boundedness inequality}
\sup_{B_\rho}g(x,|u|)\leq c \left(\frac{1}{(R-\rho)^{\frac{ps(n-1)}{s-pn}}}\left(\int_{B_R}\bigl[1+g^{\frac{n}{n-1}}(x,|u|)\bigr]\,\dd x\right)^{\frac{n-1}{n}}\right)^{\alpha \sigma},
\end{equation}
where $c$ depends on $\alpha$, $\delta$, $p$, $n$, $s$, $\sigma$ and $\|\gamma\|_{L^s(\Omega_0)}$, with $B_\rho$ and $B_R$ are the balls with center $x_0$ and radii $\rho$ and $R$ respectively.
\end{theorem}

\vskip.3cm
The main ingredients in the proofs of boundedness are, as usual, the verification that the minimizer satisfies the Euler–Lagrange system, followed by an application of Moser iteration. However, under our assumptions on g, the derivation of the Euler–Lagrange system needs an integrability condition on the local minimizer $u$, namely \eqref{gTau}.  Moreover, we prove an embedding theorem showing that $g^{\frac{n}{n-1}}(x,|u|)\in L^1$.

It is worth observing that our results generalize the $\Delta_2$-case treated in [10], where assumption (1.6) is not necessary; see also [11] for related results in the scalar case.

\vskip.3cm
 In the last twenty years, many authors have studied the regularity problem for minimizers, obtaining important results concerning boundedness and Lipschitz continuity for minimizers of integral functionals whose growth conditions are governed by different lower and upper bounds for the integrand, both in the case of distinct exponents p and q, and in settings allowing for growth that is not necessarily of power type. We remark that in  our result the framework of a possible polynomial growth is no longer adopted.

\vskip.3cm
The early influential papers \cite{Mar91,Mar93} and \cite{Talenti} for boundedness properties  have been followed by various contributions (see \cite{CupMar, DBGV} and the references therein) and at the present time the body of research on the regularity theory for local minimizers of functionals with non-standard growth conditions is so vast that providing a complete list of references is unfeasible.

Therefore, we refer the interested reader to survey articles 
on the subject \cite{Mar21}, \cite{M06} and \cite{Mingione-Radulescu} and 
we review some recent contributions on the subject:
\cite { BarCiaMar2023, BaroniColMingione, CiaSc2024, CMM2018, CMMPDN, CLM2022, CMM2023, CMM24,deFilMin1,deFilMin2, Har-Ok, Hirs-Schaf, DiMarco-Marcellini 2020, Giannone, Mar-Pacch-Nast}. For the obstacle problem we recall \cite{derosa-grimaldi, eleuteri-pass1, eleuteri-pass2}.
\vskip.3cm
The paper is organized as follows. In Section \ref{sec:Hypotheses and statement of the main result}, we give some remarks and examples. In Section \ref{euler}, we present some results on the properties of convex functionals and we derive the Euler system. In Section  \ref{sec:Immersion Theorem}, we prove an embedding theorem. Finally, in Section \ref{sec:Proof of the main theorem}, we provide the proof of the main theorem.

\section{Some remarks and examples}\label{sec:Hypotheses and statement of the main result}
 We remark a consequence of the hypothesis \eqref{ipotesi1}. For fixed \(x\) and
for every \(\lambda>1\), by convexity of
\(g(x,\cdot)\) we have
\[
g_t(x,t)(\lambda t-t)
\leq g(x,\lambda t)-g(x,t)
\leq g(x,\lambda t).
\]
Therefore,
\[
(\lambda-1)t\,g_t(x,t)
\leq g(x,\lambda t).
\]
Using \eqref{ipotesi1}, we obtain
\[
(\lambda-1)t\,g_t(x,t)
\leq C\lambda^p g^\sigma(x,t).
\]
Hence
\[
t\,g_t(x,t)
\leq C\frac{\lambda^p}{\lambda-1}g^\sigma(x,t).
\]
In particular, choosing \(\lambda=2\), we get
\begin{equation}\label{oldH2}
t\,g_t(x,t)
\leq C g^\sigma(x,t)
\leq C\bigl(1+g^\sigma(x,t)\bigr).
\end{equation}

\noindent We now give some examples of functions not in $\Delta_2$ class,  which satisfies the hypotheses \eqref{ipotesi1} and \eqref{H_3'}.
\begin{example} We remark that in the present study, we are not assuming that our function $g$ is $\Delta_2$, as it is usually done in literature (see \cite{DaMa} and references therein). 
The following example function enters in the case study we are considering, but not in previous results in literature where the $\Delta_2$ condition is assumed.\\

\noindent Let
   $$g(t)=e^{(\ln (1+t))^m} \quad \mbox{for some $t,m>1$}$$
 We have that $g$ is convex and $$\frac{g_t(t)t}{g(t)} \to \infty, \quad \mbox{as $t \to \infty$}$$
   and
   $$\frac{g(\lambda t)}{g^\sigma (t)}\to 0,\quad \mbox{with $\lambda, \sigma>1$, as $t \to \infty$}$$
Thus, $g$ is not $\Delta_2$ but it clearly satisfies the hypotheses of the present study.

\end{example}

\begin{example} This next example concerns energies with variable growth exponent. It shows that our framework includes functionals of the type\[g(x,t)=h(t)^{\alpha(x)},\] where $h$ is an $N$-function satisfying \eqref{ipotesi1}  and $\alpha \in L^\infty(\Omega)$. Here, the exponent depends on the spatial variable and the growth is not necessarily polynomial nor of \(\Delta_2\)-type. In particular, the example highlights how condition \eqref{H_3'} can still be verified despite the presence of logarithmic terms arising from the differentiation with respect to the variable exponent \(\alpha(x)\).\\

\noindent Let
\[
g(x,t)=h(t)^{\alpha(x)},
\qquad
0<c_1\leq \alpha(x)\leq c_2,
\]
and assume that for some \(p>0\), \(\sigma\geq1\), and every \(\lambda\geq1\),
\[
h(\lambda t)\leq \lambda^p h(t)^\sigma .
\]
Then
\[
g(x,\lambda t)
=
h(\lambda t)^{\alpha(x)}
\leq
\bigl[\lambda^p h(t)^\sigma\bigr]^{\alpha(x)}
=
\lambda^{p\alpha(x)}
h(t)^{\sigma\alpha(x)}.
\]
Using the bound \(\alpha(x)\leq c_2\), we obtain
\[
g(x,\lambda t)
\leq
\lambda^{pc_2}
\bigl[h(t)^{\alpha(x)}\bigr]^\sigma
=
\lambda^{pc_2}g(x,t)^\sigma .
\]
Next,
\[
g_x(x,t)
=
h(t)^{\alpha(x)}
\alpha_x(x)\log h(t),
\]
hence
\[
|g_x(x,t)|
\leq
|\alpha_x(x)|\,g(x,t)\log(1+h(t)).
\]
Since
\[
\alpha(x)\geq c_1>0,
\]
we have
\[
\frac1{\alpha(x)}
\leq
\frac1{c_1},
\]
and therefore
\[
h(t)
=
\bigl(h(t)^{\alpha(x)}\bigr)^{1/\alpha(x)}
\leq
\bigl(h(t)^{\alpha(x)}\bigr)^{1/c_1}.
\]
Consequently,
\[
\log(1+h(t))
\leq
\log\!\left(
1+
\bigl(h(t)^{\alpha(x)}\bigr)^{1/c_1}
\right).
\]
Using the elementary estimate
\[
\log(1+s)\leq C_\delta (1+s)^\delta,
\qquad s\geq0,
\]
for every \(\delta>0\), we infer that
\[
\log(1+h(t))
\leq
C_\delta
\left(
1+
\bigl(h(t)^{\alpha(x)}\bigr)^{1/c_1}
\right)^\delta.
\]
If \(c_1\leq1\), then \(1/c_1\geq1\), and hence
\[
1+
\bigl(h(t)^{\alpha(x)}\bigr)^{1/c_1}
\leq
\left(
1+h(t)^{\alpha(x)}
\right)^{1/c_1}.
\]
Therefore,
\[
\log(1+h(t))
\leq
C_\delta
\left(
1+h(t)^{\alpha(x)}
\right)^{\delta/c_1}.
\]
On the other hand, if \(c_1>1\), then \(1/c_1<1\), and
\[
\bigl(h(t)^{\alpha(x)}\bigr)^{1/c_1}
\leq
1+h(t)^{\alpha(x)}.
\]
Thus,
\[
\log(1+h(t))
\leq
C_\delta
\left(
1+h(t)^{\alpha(x)}
\right)^\delta.
\]
Combining the two cases yields
\[
\log(1+h(t))
\leq
C_\delta
\left(
1+h(t)^{\alpha(x)}
\right)^{\delta\max\{1,1/c_1\}}.
\]
Since
\[
g(x,t)=h(t)^{\alpha(x)},
\]
we conclude that
\[
\log(1+h(t))
\leq
C_\delta
(1+g(x,t))^{\delta\max\{1,1/c_1\}}.
\]
Consequently,
\[
|g_x(x,t)|
\leq
C_\delta
|\alpha_x(x)|
g(x,t)
(1+g(x,t))^{\delta\max\{1,1/c_1\}}.
\]
Finally, since
\[
g(x,t)\leq 1+g(x,t),
\]
we arrive at
\[
|g_x(x,t)|
\leq
C_\delta
|\alpha_x(x)|
(1+g(x,t))^{1+\delta\max\{1,1/c_1\}}.
\]
\end{example}

\begin{example}\label{ah} 
This example shows that our assumptions are compatible with energies having spatially dependent coefficients. More precisely, we consider integrands of the form
\[
g(x,t)=a(x)h(t),
\]
where \(a:\Omega\to(0,\infty)\) is measurable and \(h:[0,\infty)\to[0,\infty)\).
Assume that there exists $c_1$ such that
\begin{equation}\label{a-limitata}
0<c_1\leq a(x)
\end{equation} 
 and for some \(p>0\), \(\sigma\geq1\), and every \(\lambda\geq1\),
\[
h(\lambda t)\leq \lambda^p h(t)^\sigma .
\]
Then
\[
g(x,\lambda t)
=
a(x)h(\lambda t)
\leq
\lambda^p a(x)h^\sigma(t).
\]
Since
\[
a(x)
=
a(x)^\sigma a(x)^{1-\sigma}
\leq
c_1^{1-\sigma}a(x)^\sigma,
\]
it follows that
\[
g(x,\lambda t)
\leq
\lambda^p c_1^{1-\sigma}
[a(x)h(t)]^\sigma
=
\lambda^p c_1^{1-\sigma}g(x,t)^\sigma .
\]
Moreover,
\[
g_x(x,t)=a_x(x)\,h(t),
\]
and therefore
\[
|g_x(x,t)|
=
\frac{|a_x(x)|}{a(x)}\,g(x,t)
\leq
\frac{|a_x(x)|}{c_1}\,g(x,t).
\]
Since
\[
g(x,t)\leq [1+g(x,t)]^{1+\nu},
\qquad \nu>0,
\]
we conclude that
\[
|g_x(x,t)|
\leq
\frac{|a_x(x)|}{c_1}
[1+g(x,t)]^{1+\nu}.
\]

\noindent We explicitly observe that, when $h$ belongs to the $\Delta_2$-class, $\sigma =1$, assumption \eqref{a-limitata} is not necessary.
\end{example}

In the sequel, we shall use the following lemma.

\begin{lemma}[Chain rule for bounded compositions]\label{Remark2.6DallAglioMascolo}
Let $B_R\Subset \Omega$ and let $v\in W^{1,1}(B_R)$ be nonnegative and bounded.
Assume that $g$ satisfies \eqref{ipotesi1} and \eqref{H_3'}. Then
\[
g(x,v)\in W^{1,1}(B_R),
\]
and
\[
D(g(x,v))=g_x(x,v)+g_t(x,v)\,Dv
\qquad\text{a.e.\ in }B_R.
\]
Moreover,
\[
|D(g(x,v))|
\le |g_x(x,v)|+g_t(x,v)|Dv|
\qquad\text{a.e.\ in }B_R.
\]
\end{lemma}
\begin{proof}
There exists $M>0$ such that $0\le v\le M$ a.e.\ in $B_R$.
By \eqref{oldH2} and  \eqref{H_3'}, the quantities $g_t(x,v)v$ and $g_x(x,v)$ are locally integrable.
Indeed,
\[
g_t(x,v)v\le C_\sigma g^\sigma(x,v),
\]
since $g(x,\cdot)$ is increasing,
\[
g^\sigma(x,v)\le g^\sigma(x,M).
\]
Similarly,
\[
|g_x(x,v)|
\le C_\nu \gamma(x) g(x,v)\bigl(1+g^\nu(x,v)\bigr)
\le C_\nu \gamma(x) g(x,M)\bigl(1+g^\nu(x,M)\bigr),
\]
which belongs to $L^1(B_R)$.
Now, we approximate $v$ in $W^{1,1}(B_R)$ by smooth functions $v_j$ with
$0\le v_j\le M+1$. For each $j$, the classical chain rule gives
\[
D(g(x,v_j))=g_x(x,v_j)+g_t(x,v_j)Dv_j.
\]
The right-hand side is uniformly integrable by the previous estimates, so
$\{g(x,v_j)\}_j$ is bounded in $W^{1,1}(B_R)$.
Passing to the limit, we conclude that $g(x,v)\in W^{1,1}(B_R)$ and
\[
D(g(x,v))=g_x(x,v)+g_t(x,v)Dv
\qquad\textrm{a.e. in }B_R.
\]\end{proof}

\section{The Euler-Lagrange system}\label{euler}

We recall some properties of the convex function.
\begin{lemma}\label{convex1}
Let \(G:\mathbb{R}^{nN}\to[0,\infty)\) be a convex functional of class \(C^1\).
Assume that
\[
G(Du)\in L^1(\Omega)
\qquad\text{and}\qquad
G(Du+2D\varphi)\in L^1(\Omega).
\]
Then there exists a function \(h\in L^1(\Omega)\), independent of \(t\), such that
\[
\left|
\sum_{i=1}^n\sum_{\alpha=1}^N
G_i^\alpha(Du+tD\varphi)\,
\varphi^\alpha_{x_i}
\right|
\le h(x)
\]
for all \(|t|\le 1\).
\end{lemma}
\begin{proof}
By the convexity of \(G\), for every \(\xi,\xi_0\in\mathbb{R}^{nN}\),
\begin{equation}\label{convexityineq1}
G(\xi)-G(\xi_0)
\ge
D_\xi G(\xi_0)\cdot(\xi-\xi_0).
\end{equation}
Replacing \(\xi\) by \(2\xi_0-\xi\), we also obtain
\[
G(2\xi_0-\xi)-G(\xi_0)
\ge
D_\xi G(\xi_0)\cdot(\xi_0-\xi),
\]
or equivalently,
\[
G(\xi_0)-G(2\xi_0-\xi)
\le
D_\xi G(\xi_0)\cdot(\xi-\xi_0).
\]
Combining this estimate with \eqref{convexityineq1}, we get
\[
G(\xi_0)-G(2\xi_0-\xi)
\le
D_\xi G(\xi_0)\cdot(\xi-\xi_0)
\le
G(\xi)-G(\xi_0).
\]
Therefore,
\[
\bigl|
D_\xi G(\xi_0)\cdot(\xi-\xi_0)
\bigr|
\le
|G(\xi)-G(\xi_0)|
+
|G(\xi_0)-G(2\xi_0-\xi)|.
\]
Since \(G\ge0\), it follows that
\begin{equation}\label{basicconvexbound}
\bigl|
D_\xi G(\xi_0)\cdot(\xi-\xi_0)
\bigr|
\le
2G(\xi_0)+G(\xi)+G(2\xi_0-\xi).
\end{equation}
We now choose
\[
\xi_0=Du+tD\varphi,
\qquad
\xi=\xi_0+D\varphi.
\]
Then
\[
\xi-\xi_0=D\varphi,
\qquad
2\xi_0-\xi=Du+(t-1)D\varphi,
\qquad
\xi=Du+(t+1)D\varphi.
\]
Thus, from \eqref{basicconvexbound},
\begin{align}
\left|
\sum_{i=1}^n\sum_{\alpha=1}^N
G_i^\alpha(Du+tD\varphi)\,
\varphi^\alpha_{x_i}
\right|
&\le
2G(Du+tD\varphi)
+
G(Du+(t-1)D\varphi)
\nonumber\\
&\qquad
+
G(Du+(t+1)D\varphi).
\label{firstbound}
\end{align}
Since
\[
Du+tD\varphi
=
\frac12\bigl(Du+(t-1)D\varphi\bigr)
+
\frac12\bigl(Du+(t+1)D\varphi\bigr),
\]
convexity gives
\[
G(Du+tD\varphi)
\le
\frac12G(Du+(t-1)D\varphi)
+
\frac12G(Du+(t+1)D\varphi).
\]
Substituting this estimate into \eqref{firstbound}, we obtain
\begin{equation}\label{secondbound}
\left|
\sum_{i=1}^n\sum_{\alpha=1}^N
G_i^\alpha(Du+tD\varphi)\,
\varphi^\alpha_{x_i}
\right|
\le
2G(Du+(t-1)D\varphi)
+
2G(Du+(t+1)D\varphi).
\end{equation}
Let \(|t|\le1\). Then \(t-1,t+1\in[-2,2]\).
We claim that, for every \(\lambda\in[-2,2]\),
\begin{equation}\label{lambdaestimate}
G(Du+\lambda D\varphi)
\le
2G(Du)+2G(Du+2D\varphi).
\end{equation}
Indeed, by convexity,
\[
G(Du+\lambda D\varphi)
\le
\frac{2-\lambda}{4}G(Du-2D\varphi)
+
\frac{2+\lambda}{4}G(Du+2D\varphi).
\]
Moreover,
\[
G(Du)
=
G\left(
\frac12(Du-2D\varphi)+\frac12(Du+2D\varphi)
\right)
\le
\frac12G(Du-2D\varphi)
+
\frac12G(Du+2D\varphi).
\]
Hence, using \(G\ge0\),
\[
G(Du-2D\varphi)
\le
2G(Du)+2G(Du+2D\varphi).
\]
Substituting this into the previous inequality and using again that \(G\ge0\), we obtain \eqref{lambdaestimate}.
Applying \eqref{lambdaestimate} with \(\lambda=t-1\) and \(\lambda=t+1\) in \eqref{secondbound}, we get
\[
\left|
\sum_{i=1}^n\sum_{\alpha=1}^N
G_i^\alpha(Du+tD\varphi)\,
\varphi^\alpha_{x_i}
\right|
\le
8G(Du)+8G(Du+2D\varphi).
\]
Therefore, the desired estimate holds with
\[
h(x):=
8G(Du(x))+8G(Du(x)+2D\varphi(x)).
\]
By assumption, \(h\in L^1(\Omega)\). This completes the proof.
\end{proof}

For general properties on convex functions used in the present paper, we refer the reader to the book by Rao-Ren \cite{RaoRen}. 


\begin{theorem}[Euler--Lagrange system]\label{Euler-LagrangeTheorem}
Let $u$ be a local minimizer of $\mathcal{F}$ such that there exists $\tau>1$ with
\begin{equation}\label{nuova-ipotesi}
\int_{\Omega}g(x, \tau|Du|)\,\dd x < \infty .
\end{equation}
Then $u$ satisfies the Euler--Lagrange system
\begin{equation}\label{eq1.2}
\int_{\Omega} \sum_{i=1}^n F_{\xi^{\alpha}_i}(x,Du)\, \varphi^{\alpha}_{x_i}\,\dd x = 0,
\qquad \alpha = 1,\dots,N,
\end{equation}
for every $\varphi = (\varphi^{\alpha})_{\alpha=1}^N
\in V$
with $\operatorname{supp}\varphi \Subset \Omega$.
\end{theorem}

\medskip

\begin{proof}
Let $u$ be a local minimizer of $\mathcal{F}$.
Then for every $\varphi$ such that $\operatorname{supp}\varphi \Subset \Omega$
and $u+t\varphi \in V$, we have
\[
\left.\frac{d}{dt}\mathcal{F}(x,u+t\varphi)\right|_{t=0} = 0.
\]

\noindent Define
\[
\psi(x,t) := F(x,Du(x)+tD\varphi(x)).
\]
By Lemma \ref{convex1}, for sufficiently small $t$, there exists a function $h \in L^1(\Omega)$ such that
\begin{equation}\label{eq1.3}
|\psi_t(x,t)|
=
\left|
\sum_{i=1}^n \sum_{\alpha=1}^N
F^{\alpha}_i(x,Du(x)+tD\varphi(x))\, \varphi^{\alpha}_{x_i}(x)
\right|
\le h(x).
\end{equation}

\noindent By the Leibniz integral rule we obtain
\[
\frac{\dd}{\dd t}\mathcal{F}(u+t\varphi)
=
\frac{\dd}{\dd t}\int_{\Omega} F(x,Du+tD\varphi)\,\dd x
=
\int_{\Omega} \frac{\dd}{\dd t} F(x,Du+tD\varphi)\,\dd x
=
\int_{\Omega} \psi_t(x,t)\,\dd x .
\]
Hence, for $t=0$ we obtain \eqref{eq1.2}.

\noindent Now we prove \eqref{eq1.3}.
\noindent By standard calculations due to the convexity of $F$, we get
\begin{equation}\label{1.4}
\left|
\sum_{i=1}^n \sum_{\alpha=1}^N
F^{\alpha}_i(x,\xi_0)\,(\xi^{\alpha}_i-\xi^{\alpha}_{0,i})
\right|
\lesssim
2F(x,Du+(t+1)D\varphi)
+
2F(x,Du+2D\varphi).
\end{equation}

\medskip
\noindent Let us consider $\phi=\frac{\tau -1}{2\tau} \varphi$, since $\varphi \in V$ also $\phi \in V$.\\

\noindent By the convexity of $g$  we have
\begin{equation}\label{moon}
g(x,t_1+t_2)
\le
\frac1{\tau} g(x,\tau t_1)
+
\frac{\tau-1}{\tau}
g\!\left(x,\frac{\tau}{\tau-1}t_2\right),
\qquad \text{for every } \tau>1,
\end{equation}

\noindent by evaluating \eqref{moon} with $t_1=|Du|$ and $t_2=2|D\phi|$,
we explicitly obtain
\[
g(x,|Du|+2|D\phi|)
\le
\frac1{\tau} g(x,\tau|Du|)
+
\frac{\tau-1}{\tau}
g\!\left(x,\frac{\tau}{\tau-1}2|D\phi|\right)
=
\frac1{\tau} g(x,\tau|Du|)
+
\frac{\tau-1}{\tau}
g\!\left(x,\frac{2\tau}{\tau-1}|D\phi|\right).
\]

\noindent By the hypothesis \eqref{gTau}, we have that the first term on the right-hand side is integrable. Moreover,  the right-hand side of \eqref{1.4}
is locally integrable, in fact
\begin{equation}\label{1.5}
\int_{\Omega}
g\!\left(x,\frac{2\tau}{\tau-1}|D\phi|\right)\,\dd x
< \infty
\end{equation}

\noindent We have

\begin{equation}\label{eq-eulero2}
\sum_{i,\alpha}\int_{\Omega}  F_{\xi^{\alpha}_i}(x,Du)\, \left(\pm\frac{\tau-1}{2\tau}\varphi^{\alpha}_{x_i}\right)\,\dd x = 0,
\qquad \alpha = 1,\dots,N, 
\end{equation}
and then

\begin{equation}\label{eq-eulero3}
\frac{\tau-1}{2\tau} \sum_{i,\alpha}\int_{\Omega}  F_{\xi^{\alpha}_i}(x,Du)\, \varphi^{\alpha}_{x_i}\,\dd x = 0,
\qquad \alpha = 1,\dots,N, 
\end{equation}
Thus, \eqref{eq1.2} holds for every $\varphi \in V$.
\end{proof}

%
%
%

\section{Immersion Theorem}\label{sec:Immersion Theorem}
For the proof of the  of the manuscript (Theorem \eqref{boundedness theorem}), we need to verify that the right hand side of the boundedness inequality \eqref{boundedness inequality} is
finite. To this aim, we prove the following immersion Theorem.
\begin{theorem}\label{Immersion Theorem}
    Let $f(x,\xi)=g(x,|\xi|)$. Assume that $g$ satisfies  hypotheses \eqref{ipotesi1}-\eqref{H_3'}. Let $u\in W^{1,1}_{\rm{loc}}(\Omega, \mathbb{R}^N)$ such that $g(x,|Du|)\in L^1_{\rm{loc}}(\Omega)$. Then for all open sets $\Omega_0 \Subset \Omega$ there exists $R_0>0$, depending on  $n$, $\sigma$, $s$ and $\Vert \gamma\Vert_{L^s(\Omega_0)}$, such that for every $x_0\in \Omega_0$, for every $R$ such that $0<R\leq \min\{\rm{dist}(x_0, \partial\Omega_0),R_0\}$, the following inequality holds:
    \begin{equation}\label{14DallAglioMascolo}
    \left(\int_{B_R}g^{\frac{n}{n-1}}(x, |u-u_R|)\, \dd x\right)^{\frac{n-1}{n}}  \leq C \int_{B_R}g(x,|Du|)\,\dd x,
    \end{equation}
    where $u_R= \dashint_{B_R} u \,\dd x$.\\
    
\noindent As a consequence, $$g^{\sigma}(x,|u|)\in L^1_{\rm{loc}}(\Omega) \quad \mbox{with $\sigma<\frac{n}{n-1}.$}$$
\end{theorem}
\begin{proof}
   We define the truncation function
at height $k>0$ as $T_k(t)=\min \{t,k\}$. 
For $x_0\in \Omega_0$, let $0<R\leq \min\{\rm{dist}(x_0,\partial\Omega_0),1\}$. 
Since $T_k(|u-u_R|)\in W^{1,1}(B_R)$ and $0\le T_k(|u-u_R|)\le k$, Lemma \ref{Remark2.6DallAglioMascolo} yields
\[
g(x,T_k(|u-u_R|))\in W^{1,1}(B_R),
\]
and
\[
D(g(x,T_k(|u-u_R|)))=g_x(x,T_k(|u-u_R|))+g_t(x,T_k(|u-u_R|))DT_k(|u-u_R|)
\qquad\text{a.e.\ in }B_R.
\]
In particular,
\[
|D(g(x,T_k(|u-u_R|)))|
\le |g_x(x,T_k(|u-u_R|))|+g_t(x,T_k(|u-u_R|))|DT_k(|u-u_R|)|.
\]

\noindent By Sobolev's embedding
$g\left(x, T_k(|u-u_R|)\right) \in L^{\frac{n}{n-1}}(B_R)$ and we have that
\begin{align*}
   &\left(\int_{B_R}g^{\frac{n}{n-1}}(x, T_k(|u-u_R|))\, \dd x\right)^{\frac{n-1}{n}}  \\
   & \quad \leq  c(n) \left(\int_{B_R}|D(g(x,T_k(|u-u_R|))|\, \dd x + \int_{B_R} g(x,T_k(|u-u_R|))\, \dd x \right)\\
   & \quad \leq c(n)\Bigg(\int_{B_R}g_t(x,T_k(|u-u_R|))|DT_k(|u-u_R|)| \, \dd x\\
   & \quad \quad+ \int_{B_R} |g_x(x,T_k(|u-u_R|))|\, \dd x+ \int_{B_R}|g(x,T_k(|u-u_R|))| \, \dd x \Bigg)\\
   & \quad =:c(n) \,(I_1+I_2+I_3).
\end{align*}
Let $w:=T_k(|u-u_R|))$. Using Young–Fenchel inequality \eqref{YoungFenchel1} and \eqref{YoungFenchel2}, we have that
\begin{align}\label{estimategt}
g_t(x,w)\,|Dw| &\leq g^*(g_t(x,w))+g(x,|Dw|)\nonumber\\
&\leq w g_t(x,w)+g(x,|Dw|)\nonumber\\
&\leq C g^\sigma (x,w)+g(x,|Dw|).
\end{align}
By \eqref{estimategt}, we get the following estimate for the integral  $I_1$:
\begin{align*}
    I_1&\leq C\int_{B_R} g^{\sigma}(x, T_k(\vert u-u_R\vert))\dd x+ \int_{B_R} g(x, |DT_k(\vert u-u_R\vert)|)\dd x\\
&\leq C\vert B_R\vert^{\sigma\frac{n-1}{n}}\left(\int_{B_R}g^{\frac{n}{n-1}}(x, T_k(\vert u-u_R\vert))\dd x\right)^{\frac{n-1}{n}}+\int_{B_R}g(x, |DT_k(\vert u-u_k\vert )|)\dd x.
\end{align*}
Now, using hypothesis \eqref{H_3'}, we estimate $I_2$ as
\begin{align}\label{I3immersion}
    I_2&\leq C\int_{B_R}\gamma(x)\left(g^{1+\nu}(x,T_k(\vert u-u_R\vert))+1\right)\dd x \nonumber \\
    &\leq C \int_{B_R}\gamma(x)\dd x+C\Vert \gamma\Vert_{L^s(\Omega_0)}\left(\int_{B_R}g^{\frac{(1+\nu)s}{s-1}}(x,T_k(\vert u-u_R\vert))\dd x\right)^{\frac{s-1}{(1+\nu)s}}.
\end{align}
Now we can find an opportune $\nu$ such that $$\frac{(1+\nu)s}{s-1}<\frac{n}{n-1}.$$
Indeed, 
$$
\frac{(1+\nu)s}{s-1}<\frac{n}{n-1}\iff 1+\nu<\frac{n(s-1)}{s(n-1)}\iff\nu<\frac{n(s-1)}{s(n-1)}-1,
$$
which is true by the condition on $\nu$ in assumption \eqref{H_3'}.
We note that $$\nu<\frac{n(s-1)}{s(n-1)}-1<1$$ since $n\geq 2$. Indeed,
$$\frac{n(s-1)}{s(n-1)}-1<1 \iff s(n-2)+n>0.$$
We also notice that, being $s>pn>n$ by hypothesis \eqref{H_3'}, we have
$$
\frac{n(s-1)}{s(n-1)}-1= \frac{s-n}{s(n-1)}>0.
$$
Thus, the last factor in the right hand side of inequality \eqref{I3immersion} can be estimated as follows
\begin{align*}
&\left(\int_{B_R}g^{\frac{(1+\nu)s}{s-1}}(x,T_k(\vert u-u_R\vert))\dd x\right)^{\frac{s-1}{(1+\nu)s}}\\
&\quad \leq\left(\int_{B_R}g^
   {\frac{n}{n-1}}(x,T_k(\vert u-u_R\vert))\dd x\right)^{\frac{n-1}{n}}\vert B_R\vert^{\frac{s-1}{(1+\nu)s}-\frac{n-1}{n}}.
\end{align*}
Finally, we estimate the third integral $I_3$ by using H\"{o}lder inequality:
\begin{equation*}
    I_3\leq \vert B_R\vert^{\frac{1}{n}}\left(\int_{B_R}g^{\frac{n}{n-1}}(x, T_k(\vert u-u_R\vert))\dd x\right)^{\frac{n-1}{n}}.
\end{equation*}
By putting together the above estimates on the three integrals, we have proved that
\begin{align*}
    &\left(\int_{B_R}g^{\frac{n}{n-1}}(x,T_k(|u-u_R|))\, \dd x\right)^{\frac{n-1}{n}}\\
    &\quad\leq c(n) \Bigg(C|B_R|^{\sigma\frac{ n-1}{n}}\left(\int_{B_R}g^{\frac{n}{n-1}}(x,T_k(|u-u_R|))\, \dd x\right)^{\frac{n-1}{n}}+\int_{B_R}g(x, |DT_k(\vert u-u_k\vert )|)\dd x\\
    &\quad\quad+
    |B_R|^{\frac{1}{n}}\left(\int_{B_R}g^{\frac{n}{n-1}}(x,T_k(|u-u_R|))\, \dd x\right)^{\frac{n-1}{n}}\\
    &\quad \quad +\int_{B_R}\gamma(x)\dd x+\Vert \gamma\Vert_{L^s(\Omega_0)}\left(\int_{B_R}g^
   {\frac{n}{n-1}}(x,T_k(\vert u-u_R\vert))\dd x\right)^{\frac{n-1}{n}}\vert B_R\vert^{\frac{s-1}{(1+\nu)s}-\frac{n-1}{n}}\Bigg)\\
   &= c(n) \left(\int_{B_R}g^{\frac{n}{n-1}}(x,T_k(|u-u_R|))\, \dd x\right)^{\frac{n-1}{n}}\Bigg(C|B_R|^{\sigma\frac{ n-1}{n}}+
    |B_R|^{\frac{1}{n}}+\Vert \gamma\Vert_{L^s(\Omega_0)}\vert B_R\vert^{\frac{s-1}{(1+\nu)s}-\frac{n-1}{n}}\Bigg) \\
    &\quad \quad +c(n)\int_{B_R}g(x, |DT_k(\vert u-u_k\vert )|)\dd x+c(n) \int_{B_R}\gamma(x)\dd x\\
    &\quad \leq c_1 \left(\int_{B_R}g^{\frac{n}{n-1}}(x,T_k(|u-u_R|))\, \dd x\right)^{\frac{n-1}{n}}\Bigg(C|B_R|^{\sigma\frac{ n-1}{n}}+
    |B_R|^{\frac{1}{n}}+\Vert \gamma\Vert_{L^s(\Omega_0)}\vert B_R\vert^{\frac{s-1}{(1+\nu)s}-\frac{n-1}{n}}\Bigg) \\
    &\quad \quad +c_1\int_{B_R}g(x, |DT_k(\vert u-u_k\vert )|)\dd x,
\end{align*}
where in the last step we have used that  $c(n)\int_{B_R}\gamma(x)\dd x>0$ is constant.
By choosing $R_0$ sufficiently small, depending on $n$, $\sigma$, $s$ and $\Vert \gamma\Vert_{L^s(\Omega_0)}$, so that for $R\leq \min\{\rm{dist}(x_0, \partial\Omega_0),R_0\}$, we get that
$$C|B_R|^{\sigma\frac{ n-1}{n}}+
    |B_R|^{\frac{1}{n}}+\Vert \gamma\Vert_{L^s(\Omega_0)}\vert B_R\vert^{\frac{s-1}{(1+\nu)s}-\frac{n-1}{n}} \leq \frac{1}{2c_1}.$$
As a consequence, we have that 
\begin{equation*}
   \left(\int_{B_R}g^{\frac{n}{n-1}}(x,T_k(|u-u_R|))\, \dd x\right)^{\frac{n-1}{n}} \leq 2c_1  \int_{B_R}g(x, |DT_k(\vert u-u_k\vert )|)\dd x,
\end{equation*}
and by letting $k$ go to infinity, we obtain \eqref{14DallAglioMascolo}.
\end{proof}

\section{Proof of the main theorem}\label{sec:Proof of the main theorem}\label{Section4} 
\begin{proof}[Proof of Theorem \ref{boundedness theorem}]
 Let $u$ be a local minimizer of $\mathcal{F}$ there exists $\tau$ such that:
\[
\int_{\Omega} g(x,\tau |Du|)\,\dd x < \infty
\quad \text{for some } \tau>1.
\]

\noindent Then, by Theorem \ref{Euler-LagrangeTheorem}, $u$ satisfies the weak form of the Euler system,
namely 
\begin{equation}\label{1.6}
\int_{\Omega}
\sum_{i=1}^n F_i^\alpha(x,Du)\, \Phi_{x_i}^\alpha \, \dd x = 0,
\qquad \alpha=1,\dots,N.
\end{equation}\noindent for all $\Phi \in V.$\\

 Let  $\varphi$ be satisfy the following properties:
\begin{itemize}
    \item [$(\varphi_1)$] $\varphi:\Omega \times [0,+\infty) \to [0,+\infty)$ is a bounded Caratheodory function;
    \item [$(\varphi_2)$] the derivative $\varphi_t(x, t)$ with respect to $t$ is a nonnegative bounded Caratheodory
function such that $\varphi_t(x, t)t$ is bounded, \(\varphi(x,t)\) and \(\varphi_t(x,t)t\) are bounded by a constant \(M>1\);
    \item [$(\varphi_3)$] for all $t\geq 0$, $\varphi(\cdot, t)$ has weak derivatives in $L^1_{\rm{loc}}(\Omega)$, which are Caratheodory function on $\Omega \times [0,+\infty)$, and are such that there exists $\omega \in L^s_{\rm{loc}}(\Omega)$, with $s>pn$ such that 
    \begin{equation}\label{bound_varphi_x}
       |\varphi_x(x,t)|\leq \omega(x) \quad \mbox{for every $t \geq 0$, for a.e. $x\in \Omega$.}
    \end{equation}
    It is not restrictive to suppose that $\omega(x)>1$ for a.e. $x\in \Omega$.
\end{itemize}

\noindent  Define the function $\Phi$:
\[
\Phi^\alpha(x) = \frac{\delta}{M}u^\alpha \, \varphi(x,|\delta u|)\, \eta^p, 
\]
where $0<\delta<1$ opportunely chosen later in the proof, 
$\eta \in C_0^1(B_R)$ satisfies 
$$
0\le \eta \le 1,\quad
\eta \equiv 1 \mbox{ in $B_\rho$}, \quad |D\eta| \le \frac{1}{R-\rho},
$$
\noindent By definition, $\Phi \in W^{1,1}(B_R)$ and the chain rule of differentiation gives the following derivatives of $\Phi^\alpha$:
\begin{align}\label{derivativePhi}
\Phi^\alpha_{x_i}
=&
\frac{1}{M}\big(p\,\eta^{p-1}(\partial_{x_i}\eta)\,\varphi(x,\delta| u|)\,\delta u^\alpha
+
\eta^p\,\varphi(x,\delta| u|)\,\delta u^\alpha_{x_i}\nonumber \\
&+\eta^p\,\varphi_t(x,\delta | u|)
\frac{\delta^2 u^\alpha}{| u|}
\sum_{\beta=1}^N  u^\beta  u^\beta_{x_i}+
\eta^p\,\varphi_{x_i}(x,\delta | u|)\,\delta u^\alpha\big).
\end{align}
\noindent We claim that \(g(x,|D\Phi|)\in L^1_{\rm{loc}}(\Omega)\). Indeed, by \eqref{derivativePhi}, recalling that
\(\varphi(x,t)\) and \(\varphi_t(x,t)t\) are bounded by a constant \(M>1\), using that
\(0\leq \eta\leq 1\) and \eqref{bound_varphi_x}, we have
\begin{align*}
|D\Phi|
&\leq
 \frac{1}{M}p\,\eta^{p-1}|D\eta|\,\varphi(x,\delta |u|)\,\delta |u|
+\frac{1}{M}\eta^p\varphi(x,\delta |u|)\,\delta |Du| \\
&\quad
+\frac{1}{M}\eta^p\varphi_t(x,\delta |u|)\delta^2 |u|\,|Du|
+\frac{1}{M}\eta^p|\varphi_x(x,\delta |u|)|\,\delta |u|  \\
&\leq
\frac{p\delta}{R-\rho}|u|
+2\delta |Du|
+\frac{1}{M}\omega(x)\delta |u|  \\
&=: A_1+A_2+A_3 .
\end{align*}
Hence, using \eqref{moon} first with parameter \(2\), and then with the parameter \(\tau>1\), we obtain
\begin{align}\label{goodtestfunction}
g(x,|D\Phi|)
&\leq
\frac12 g\bigl(x,2(A_1+A_2)\bigr)
+\frac12 g(x,2A_3)\nonumber \\
&\leq
\frac{1}{2\tau}g(x,4\tau \delta |Du|)
+\frac12\frac{\tau-1}{\tau}
g\left(x,\frac{2\tau}{\tau-1}\frac{p\delta}{R-\rho}|u|\right)
+\frac12 g\left(x,\frac{1}{M}2\delta\omega(x)|u|\right).
\end{align}

\noindent We now estimate the three terms on the right-hand side. By assumption, there exists
\(\tau>1\) such that
\[
\int_{\Omega} g(x,\tau |Du|)\,\dd x<\infty.
\]
We choose
\begin{equation}\label{condition1delta}
0<\delta\leq \frac{1}{4}.
\end{equation}
Then \(4\delta\tau\leq \tau\), and by the monotonicity of \(g(x,\cdot)\),
\[
g(x,4\tau \delta |Du|)
\leq g(x,\tau |Du|).
\]
Therefore,
\[
g(x,4\tau \delta |Du|)\in L^1_{\rm{loc}}(\Omega).
\]

\noindent We next estimate the second term. 
Using \eqref{ipotesi1} we have
$$g\left(x,\frac{2\tau}{\tau-1}\frac{p\delta}{R-\rho}|u|\right) \leq g\left(x,\left(1+\frac{2\tau}{\tau-1}\frac{p\delta}{R-\rho}\right)|u|\right)\leq \left(1+\frac{2\tau}{\tau-1}\frac{p\delta}{R-\rho}\right)^p\, g^\sigma(x,|u|).$$

\noindent Since, by Theorem \ref{Immersion Theorem},
\[g^\sigma(x,|u|)\in L^1_{\mathrm{loc}}(\Omega),
\]
we obtain
\[
g\left(x,\frac{2\tau}{\tau-1}\frac{p\delta}{R-\rho}|u|\right)
\in L^1_{\rm{loc}}(\Omega).
\]

\noindent Lastly, we estimate the term $g\left(x,\frac{1}{M}2\delta\omega(x)|u|\right)$. We apply hypothesis \eqref{ipotesi1} and we get

$$g\left(x,\frac{1}{M}2\delta\omega(x)|u|\right)\leq 1+\frac{1}{M}2\delta\omega(x))^p \, g^\sigma (x,|u|)$$

\noindent By assumption,
\(\omega\in L^s_{\mathrm{loc}}(\Omega)\), with \(s>pn\), and therefore
$\omega^p\in L^{\frac{s}{p}}_{\rm{loc}}(\Omega)$. By Hölder's inequality,
\[
\int_{B_R}\left((1+\frac{1}{M}2\delta\omega(x)\right)^pg^\sigma(x,|u|)\,\dd x
\leq
\left(\int_{B_R}\left(1+\frac{1}{M}2\delta\omega(x)\right)^s\,\dd x\right)^{\frac{p}{s}}
\left(\int_{B_R}\left(g^\sigma(x,|u|)\right)^{\frac{s}{s-p}}\,\dd x\right)^{\frac{s-p}{s}}.
\]
Thus, it is enough to have $g^\sigma(x,|u|)
\in L^{\frac{s}{s-p}}_{\rm{loc}}(\Omega)$,
or equivalently,
\begin{equation}\label{tobeiintegrable}
g(x,|u|)
\in L^{\sigma\frac{s}{s-p}}_{\rm{loc}}(\Omega). \end{equation}
By Theorem \ref{Immersion Theorem}, $g(x,|u|)\in L^{\frac{n}{n-1}}(B_R)$,
hence \eqref{tobeiintegrable} is satisfied provided that
\[
\sigma\frac{s}{s-p}\leq \frac{n}{n-1},
\]
that is true by hypothesis \eqref{boundsigma}.\\

\noindent Under this condition on $\sigma$, we conclude $g\left(x,\frac{1}{M}2\delta\omega(x)|u|\right)
\in L^1_{\rm{loc}}(\Omega)$.

\noindent We have shown that each term on the right-hand side of \eqref{goodtestfunction} belongs to \(L^1_{\rm{loc}}(\Omega)\). Hence
\[
g(x,|D\Phi|)\in L^1_{\rm{loc}}(\Omega),
\]
that is \(\Phi\) is an admissible test function in \eqref{1.6}.\\

\noindent Taking into account that
\[
F^\alpha_i(x,\xi)
=
\frac{\partial F(x,\xi)}{\partial \xi^\alpha_i}
=
g_t(x,|\xi|)\,\frac{\xi^\alpha_i}{|\xi|},
\]
we may write, directly in terms of $Du$,
\[
F^\alpha_i(x,Du)
=
g_t(x,|Du|)\,\frac{u^\alpha_{x_i}}{|Du|}.
\]

\noindent Finally, summing over $\alpha$ in \eqref{1.6}, we can rewrite the Euler
system as
\begin{equation}\label{1.7}
0 = \frac{1}{M} (I_1 + I_2 + I_3+I_4), \quad \mbox{ that is }
0 = I_1 + I_2 + I_3+I_4 ,
\end{equation}
where
\begin{align*}
I_1
&=
\int_{\Omega}
\sum_{i,\alpha}
\frac{g_t(x,|Du|)}{|Du|}
\,u^\alpha_{x_i}
\bigl(
p\,\eta^{p-1}\eta_{x_i}\,\delta u^\alpha\,\varphi(x,\delta| u|)
\bigr)\,\dd x ,
\\[0.3em]
I_2
&=
\int_{\Omega}
\sum_{i,\alpha}
\frac{g_t(x,|Du|)}{|Du|}
\,u^\alpha_{x_i}
\bigl(
\eta^p\,\delta u^\alpha_{x_i}\,\varphi(x,\delta |u|)
\bigr)\,\dd x ,
\\[0.3em]
I_3
&=
\int_{\Omega}
\sum_{i,\alpha}
\frac{g_t(x,|Du|)}{|Du|}
\,u^\alpha_{x_i}
\biggl(
\eta^p\,\delta^2 u^\alpha\,\varphi'(x,\delta |u|)
\frac{1}{|u|}
\sum_{\beta}
u^\beta u^\beta_{x_i}
\biggr)\,\dd x,
\\[0.3em]
I_4
&=
\int_{\Omega}
\sum_{i,\alpha}
\frac{g_t(x,|Du|)}{|Du|}
\,u^\alpha_{x_i}\, \biggl(\eta^p
\,\delta u^\alpha\,\varphi_{x_i}(x,\delta| u|)\biggr)\,\dd x.
\end{align*}
Rearranging, we obtain
\begin{align*}
I_1 + I_2 + I_3 + I_4
&=
\int_{\Omega}
p\delta\frac{g_t(x,|Du|)}{|Du|}
\,\eta^{p-1}\varphi(x,\delta |u|)
\sum_{i,\alpha}
\eta_{x_i} u^\alpha_{x_i} u^\alpha \,\dd x
\\
&\quad
+ \int_{\Omega}
\delta g_t(x,|Du|)\,|Du|\,
\eta^p \varphi(x,\delta |u|)\,\dd x
\\
&\quad
+
\int_{\Omega}
\eta^p\delta^2
\frac{g_t(x,|Du|)}{|Du|}
\,\varphi'(x,\delta |u|)
\frac{1}{| u|}
\sum_{i}
\left(
\sum_{\alpha} u^\alpha u^\alpha_{x_i}
\right)^2
\dd x\\
&\quad
+  \int_{\Omega}\eta^p\,\delta \,
\frac{g_t(x,|Du|)}{|Du|}
\,\sum_{i,\alpha} u^\alpha_{x_i}\, 
u^\alpha\,\varphi_{x_i}(x,\delta| u|)\,\dd x =0.
\end{align*}
For the third term we have
\begin{align*}
I_3
&=
\int_{\Omega}
\sum_{i,\alpha}
\frac{g_t(x,|Du|)}{|Du|}
\,u^\alpha_{x_i}
\biggl(
\eta^p\,\delta^2 u^\alpha\,\varphi'(x,\delta |u|)
\frac{1}{|u|}
\sum_{\beta} u^\beta u^\beta_{x_i}
\biggr)\,\dd x
\\
&=
\int_{\Omega}
\eta^p\delta^2
\frac{g_t(x,|Du|)}{|Du|}
\frac{\varphi'(x,\delta |u|)}{| u|}
\sum_{i}
\sum_{\alpha,\beta}
u^\alpha_{x_i} u^\alpha\,
u^\beta u^\beta_{x_i}\,\dd x
\\
&=
\int_{\Omega}
\eta^p\delta^2
\frac{g_t(x,|Du|)}{|Du|}
\frac{\varphi'(x,\delta |u|)}{|u|}
\sum_{i}
\left(
\sum_{\alpha} u^\alpha u^\alpha_{x_i}
\right)
\left(
\sum_{\beta}  u^\beta u^\beta_{x_i}
\right)\,\dd x
\\
&=
\int_{\Omega}
\eta^p\delta^2
\frac{g_t(x,|Du|)}{|Du|}
\frac{\varphi'(x,\delta |u|)}{|u|}
\sum_{i}
\bigl(|u|(|u|)_{x_i}\bigr)^2
\,\dd x .
\end{align*}
Since the integral $I_3$ is positive or equal to zero, by \eqref{1.7} we obtain
\[
0 = I_1 + I_2 + I_3 +I_4\ge I_1 + I_2+I_4,
\qquad\Rightarrow\qquad
I_2 = -(I_1+I_3+I_4) \le -I_1-I_4,
\]
and therefore
\[
I_2 \le |I_1|+|I_4|.
\]
We also note that
\[
\left|
\sum_{i,\alpha}\eta_{x_i}u^\alpha_{x_i}\delta u^\alpha
\right|
\le
\sum_{i,\alpha}|\eta_{x_i}|\,|u^\alpha_{x_i}|\,|\delta u^\alpha|
\le
|D\eta|\,|\delta u|\,|Du|,
\]
so that we have the following estimate for $I_1$
\begin{align*}
|I_1|
\le &
\int_{\Omega}
\frac{g_t(x,|Du|)}{|Du|}
\,p\,\eta^{p-1}\,\varphi(x,\delta |u|)
\left|
\sum_{i,\alpha}\eta_{x_i}u^\alpha_{x_i}\delta u^\alpha
\right|\,\dd x \\ \le &
\, p 
\int_{\Omega}
\frac{g_t(x,|Du|)}{|Du|}
\,\varphi(x,\delta |u|)\,\eta^{p-1}
|D\eta|\,|\delta u|\,|Du|\,\dd x
\end{align*}
Analogously, we can estimate $I_4$. Thus, we get
\begin{align}\label{1.8_2}
\delta\int_{\Omega} g_t(x,|Du|)\,|Du|\,\eta^p\,\varphi(x,\delta |u|)\,\dd x
\le & \,p
\int_{\Omega}
g_t(x,|Du|)\,\varphi(x,\delta |u|)\,\eta^{p-1}
|D\eta|\,|\delta u|\,\dd x \nonumber \\
&+\int_{\Omega}g_t(x,|Du|)\,|\varphi_x(x,\delta |u|)|\,\eta^{p} \,|\delta u|\,\dd x.
\end{align}

d moreover $\varphi_k'$ is bounded.
%
%
%

%

\medskip

\noindent Let now $\beta$ be a positive constant and let  $(h_k)_{k\in\mathbb{N}}\subset C^1([0,\infty))$ be a sequence of
nonnegative, bounded, increasing functions with respect to $k$ such that:
\begin{itemize}
\item $h_k$ is constant for $t$ large enough (so $h_k$ is bounded);
\item $h_k(t)\uparrow t^{\beta}$ for every $t\ge0$;
\item $h_k'(t)\uparrow \beta t^{\beta-1}$ for every $t>0$.
\end{itemize}

\noindent Define
\[
\varphi_k(x,t)=h_k(g(x,t)), \qquad t\ge0.
\]

\noindent Then each $\varphi_k\in C^1([0,\infty))$ is convex and increasing,
satisfies $\varphi_k(0)=0$, and $(\varphi_1)-(\varphi_3)$ has bounded derivative. Hence $\varphi_k$ is admissible in \eqref{1.6} and by 
\eqref{1.8_2} with $\varphi=\varphi_k$  we get
\begin{align*}
\delta\int_{\Omega} \eta^p \, g_t(x,|Du|)\,|Du|\,h_k\bigl(g(x,|\delta u|)\bigr)\,\dd x
\le & \,p
\int_{\Omega}\eta^{p-1}\,
g_t(x,|Du|)\,h_k\bigl(g(x,|\delta u|)\bigr)\,
|D\eta|\,|\delta u|\,\dd x \nonumber \\
&+\int_{\Omega}\eta^{p}\, g_t(x,|Du|)\,h'_k\bigl(g(x,|\delta u|)\bigr)\,|g_x(x,|\delta u|)|\,|\delta u|\,\dd x.
\end{align*}
Since
\[
h_k\bigl(g(x,|\delta u|)\bigr)\uparrow g^{\beta}(x,|\delta u|)
\quad\text{pointwise}
\]
and all integrands are nonnegative, we may let $k\to\infty$
and apply the Monotone Convergence Theorem and hypothesis \eqref{H_3'} to obtain
\begin{align}\label{1.9_2}
\delta\int_{\Omega} \eta^p \, g_t(x,|Du|)\,|Du|\,g^{\beta}(x,|\delta u|)\,\dd x
&\le  \,p
\int_{\Omega}\eta^{p-1}\,
g_t(x,|Du|)\,g^{\beta}(x,|\delta u|)\,
|D\eta|\,|\delta u|\,\dd x \nonumber \\
&\quad +C \, \beta \int_{\Omega}\eta^{p}\, g_t(x,|Du|)g^{\beta}(x,|\delta u|)\bigl(1+g^{\nu}(x,|\delta u|)\bigr)\,\gamma(x)\,|\delta u|\,\dd x\nonumber\\
&=A+B.
\end{align}
We start by estimating the first integral on the right hand side $A$.
We note that if $g:\mathbb{R}\times[0,\infty)\to[0,\infty)$ is a convex, increasing function, and 
$g^*$ is its convex conjugate,
\[
g^*(x,s) := \sup_{t\ge 0}\{st - g(x,t)\}.
\]
Then the Young-Fenchel inequality states that for all $a,b\ge 0$,
\begin{equation*}
ab \le g^*(\cdot,a) + g(\cdot,b).
\end{equation*}
A rescaled version is
\begin{equation}\label{YoungFenchel1}
ab \le \varepsilon g^*(\cdot,a) + \varepsilon g\!\left(\cdot,\frac{b}{\varepsilon}\right),
\qquad \varepsilon>0.
\end{equation}

\noindent Moreover, by the convexity of $g$, one has
\begin{equation}\label{YoungFenchel2}
g^*(x,g_t(x,t)) = t g_t(x,t) - g(x,t) \le t g_t(x,t).
\end{equation}
In particular,
\[
g^*(x,g_t(x,|Du|)) \le g_t(x,|Du|)\,|Du|.
\]

\noindent For $x$ such that $\eta(x)\neq 0$, multiplying 
\eqref{YoungFenchel2} by $\eta^p$ (used with $a=g_t(x,|Du|)$ and $b=|\delta u|\,|D\eta|/\eta$) and using the previous inequality, we obtain
\begin{align}\label{smileyFace}
  \eta^{p-1} g_t(x,|Du|)\,|\delta u|\,|D\eta|
&\le
\varepsilon\,\eta^p g^*\!\bigl(x,g_t(x,|Du|)\bigr)
+
\varepsilon\,\eta^p
g\!\left(x,\frac{|\delta u|\,|D\eta|}{\varepsilon\eta}\right)\nonumber\\  
&\le
\varepsilon\,\eta^p g_t(x,|Du|)\,|Du|
+
\varepsilon\,\eta^p
g\!\left(x,\frac{|\delta u|\,|D\eta|}{\varepsilon\eta}\right).
\end{align}
We now focus on the term $g\!\left(x,\frac{|\delta u|\,|D\eta|}{\varepsilon\eta}\right)$.
Since $|D\eta| \le \frac{1}{R-\rho}$, and because $g$ satisfies \eqref{ipotesi1}, we have
\[
g(x,\lambda t) \le C \lambda^p \bigl[1+g^\sigma(x,t)\bigr],
\qquad \mbox{for every $t>0$.} 
\]

\noindent Thus, choosing
\[
\lambda=\frac{|D\eta|}{\varepsilon\eta},
\qquad
t=|\delta u|,
\]
we obtain
\begin{align*}
\eta^p g\!\left(x,\frac{|\delta u|\,|D\eta|}{\varepsilon\eta}\right)
&\le
\eta^p\left(\frac{C}{(R-\rho)^p}\,\varepsilon^{-p}\eta^{-p}
\bigl[1+g^\sigma(x,|\delta u|)\bigr]\right)\\
&=\frac{C}{(R-\rho)^p}\,\varepsilon^{-p}
\bigl[1+g^\sigma(x,|\delta u|)\bigr].
\end{align*}

\medskip

\noindent Returning to \eqref{smileyFace}, we conclude that
\[
\eta^{p-1} g_t(x,|Du|)\,|\delta u|\,|D\eta|
\le
\varepsilon\,\eta^p g_t(x,|Du|)\,|Du|
+
\frac{C\,\varepsilon^{1-p}}{(R-\rho)^p}
\bigl[1+g^\sigma(x,|\delta u|)\bigr].
\]
Thus, we get 
\begin{align}\label{estimateA}
A:=p
\int_{B_R}\eta^{p-1}\,
g_t(x,|Du|)\,g^{\beta}(x,|\delta u|)\,
|D\eta|\,|\delta u|\,\dd x \leq & \, p \, \varepsilon \int_{B_R} g^{\beta}(x,|\delta u|)\,\eta^p g_t(x,|Du|)\,|Du|\,\dd x\nonumber \\ &+ p\, \frac{C\,\varepsilon^{1-p}}{(R-\rho)^p}  \int_{B_R} g^{\beta}(x,|\delta u|)\, 
\bigl[1+g^\sigma(x,|\delta u|)\bigr]\,\dd x.
\end{align}
We now estimate the second integral on the right hand side \begin{align*}
    B&:=C \, \beta \int_{B_R}\eta^{p}\, g_t(x,|Du|)g^{\beta}(x,|\delta u|)\bigl(1+g^{\nu}(x,|\delta u|)\bigr)\,\gamma(x)\,|\delta u|\,\dd x.
    \end{align*}
We use \eqref{YoungFenchel1} with $a=g_t(x,|Du|)$ and $b=\bigl(1+g^{ \nu}(x,|\delta u|)\bigr)\,\gamma(x)|\delta u|$, we multiply it by $\eta^p$, so we obtain
\begin{align}\label{smileyFacebis}
  \eta^{p} g_t(x,|Du|)\bigl(1+g^{ \nu}(x,|\delta u|)\bigr)\,\gamma(x)|\delta u|
&\le
\varepsilon\,\eta^p g^*\!\bigl(x,g_t(x,|Du|)\bigr)
+
\varepsilon\,\eta^p
g\!\left(x,\frac{\bigl(1+g^{ \nu}(x,|\delta u|)\bigr)\,\gamma(x)|\delta u|}{\epsilon}\right)\nonumber\\  
&\le
\varepsilon\,\eta^p g_t(x,|Du|)\,|Du|
+
\varepsilon\,\eta^p
g\!\left(x,\frac{\bigl(1+g^{ \nu}(x,|\delta u|)\bigr)\,\gamma(x)|\delta u|}{\epsilon}\right).
\end{align}
We now focus on the term $g\!\left(x,\frac{\bigl(1+g^{ \nu}(x,|\delta u|)\bigr)\,\gamma(x)|\delta u|}{\epsilon}\right)$.
Since $g$ satisfies \eqref{ipotesi1}, we have
\[
g(x,\lambda t) \le C \lambda^p \bigl[1+g^\sigma(x,t)\bigr],
\qquad \mbox{for every $t>0$.} 
\]
Thus, choosing
\[
\lambda=\frac{\bigl(1+g^{ \nu}(x,|\delta u|)\bigr)\,\gamma(x)}{\epsilon},
\qquad
t=|\delta u|,
\]
we obtain
\[
g\!\left(x,\frac{\bigl(1+g^{ \nu}(x,|\delta u|)\bigr)\,\gamma(x)|\delta u|}{\epsilon}\right)
\le C\left(\frac{\bigl(1+g^{ \nu}(x,|\delta u|)\bigr)\,\gamma(x)}{\epsilon}\right)^p
\bigl[1+g^\sigma(x,|\delta u|)\bigr].
\]
Returning to \eqref{smileyFacebis}, we have
\begin{align*}
 \eta^{p} g_t(x,|Du|)\bigl(1+g^{ \nu}(x,|\delta u|)\bigr)\,\gamma(x)|\delta u|
&\leq
\varepsilon\,\eta^p g_t(x,|Du|)\,|Du|\\
&\quad+
\varepsilon\,\eta^p C\left(\frac{\bigl(1+g^{ \nu}(x,|\delta u|)\bigr)\,\gamma(x)}{\epsilon}\right)^p
\bigl[1+g^\sigma(x,|\delta u|)\bigr]\\
&=
\varepsilon\,\eta^p g_t(x,|Du|)\,|Du|\\
&\quad+
\varepsilon^{1-p}\,\eta^p C\bigl(1+g^{ \nu}(x,|\delta u|)\bigr)^p\,\gamma^p(x)
\bigl[1+g^\sigma(x,|\delta u|)\bigr].
\end{align*}
Then we deduce the following estimate for integral $B$
\begin{align}\label{estimateB}
B&\leq C \, \beta \varepsilon\,\int_{B_R} \eta^p g^{\beta}(x,|\delta u|)\,g_t(x,|Du|)\,|Du|\, \dd x \nonumber \\
&\quad +C \, \beta \varepsilon^{1-p}\int_{B_R} \eta^p g^{\beta}(x,|\delta u|)\,\bigl(1+g^{ \nu}(x,|\delta u|)\bigr)^p\,\gamma^p(x)
\bigl[1+g^\sigma(x,|\delta u|)\bigr]\, \dd x.
\end{align}
Putting together \eqref{1.9_2},\eqref{estimateA} and \eqref{estimateB}, we obtain
\begin{align*}
\delta\int_{B_R} \eta^p \, g_t(x,|Du|)\,|Du|\,g^{\beta}(x,|\delta u|)\,\dd x
&\le \, p \, \varepsilon \int_{B_R}\eta^p g^{\beta}(x,|\delta u|)\, g_t(x,|Du|)\,|Du|\,\dd x\nonumber \\ 
& \quad + p\, \frac{C\,\varepsilon^{1-p}}{(R-\rho)^p}  \int_{B_R} g^{\beta}(x,|\delta u|)\, 
\bigl[1+g^\sigma(x,|\delta u|)\bigr]\,\dd x\nonumber \\
&\quad+C \, \beta \varepsilon\,\int_{B_R} \eta^p g^{\beta}(x,|\delta u|)\,g_t(x,|Du|)\,|Du|\, \dd x \nonumber \\
&\quad +C \, \beta \varepsilon^{1-p}\, C\int_{B_R} \eta^p g^{\beta}(x,|\delta u|)\,\bigl(1+g^{ \nu}(x,|\delta u|)\bigr)^p\,\gamma^p(x)
\bigl[1+g^\sigma(x,|\delta u|)\bigr]\, \dd x.
\end{align*}
Thus, by choosing opportunely $\epsilon$ to absorb the first and the third addendum of the right hand side into the left hand side and recalling that $0\leq \eta\leq1$, we get
\begin{align}\label{(24)DallAglioMascolo}
\int_{B_R} \delta \, \eta^p \, g_t(x,|Du|)\,|Du|\,g^{\beta}(x,|\delta u|)\,\dd x
&\le \frac{c}{(R-\rho)^p} \int_{B_R}   \gamma^p(x)\,\bigl(1+g^{ \nu p +\beta+\sigma}(x,|\delta u|)\bigr)\, \dd x.
\end{align}
where $c$ depends on $\nu$, $\beta$, $p$, $\sigma$, $\delta$.\\

\noindent Let us define 
\begin{equation}\label{definition_w}
    w := \eta^p \bigl[g^{\beta+1}(x,|\delta u|)+1\bigr].
\end{equation}
Formally, computing the derivative, we obtain
\begin{align}\label{2.5DallAglioMascolo}
w_{x_i}=& \,p\eta^{p-1}\eta_{x_i}\bigl[g^{\beta+1}(x,|\delta u|)+1\bigr]+(\beta +1) \eta^p g^{\beta}(x,|\delta u|) g_{x_i}(x,|\delta u|)\nonumber\\&+ (\beta+1) \eta^p g^{\beta}(x,|\delta u|) \frac{g_t(x,|\delta u|)}{|u|}\,\delta\,\sum_{\alpha=1}^N u^\alpha u_{x_i}^\alpha.
\end{align}
Assuming that the right-hand side belongs to $L^1(\Omega)$, we conclude that $w\in W^{1,1}_0(\Omega)$ and that the above expression coincides with its weak gradient.\\


\noindent Using \eqref{2.5DallAglioMascolo}, and in particular hypothesis \eqref{H_3'} to estimate the second addendum on the right hand side of \eqref{2.5DallAglioMascolo}, we get
\begin{align}\label{doubleSmile}
\int_{B_\rho} |Dw|\,\dd x
\le &
\,p\int_{B_R}
\eta^{p-1}[g^{\beta+1}(x,|\delta u|)+1]|D\eta|\,\dd x 
+ (\beta+1) C \int_{B_R} \eta^p \gamma(x) [g^{\beta+1+\nu}(x,|\delta u|)+1]\,\dd x
\nonumber\\
&+ 
(\beta+1)
\int_{B_R}
\eta^p g^{\beta}(x,|\delta u|)
\,\delta\,|Du|\, g_t(x,|\delta u|)\,\dd x.
\end{align}
Now, we want to estimate the third term on the right-hand side of inequality \eqref{doubleSmile}.
We observe that, again by standard inequality $g_t(\cdot,t_1)t_2\leq g_t(\cdot,t_1)t_1+g_t(\cdot,t_2)t_2$, we have that
\begin{equation*}
g_t(x,|\delta u|)|Du| \le
g_t(x,|\delta u|)|\delta u|
+
g_t(x,|Du|)|Du|.
\end{equation*}
Then the third integral on the right hand side of \eqref{doubleSmile} can be estimated as
\begin{align}\label{estimate_time_derivative_term}
\int_{B_R}
 \eta^p g^{\beta}(x,|\delta u|)
\,\delta\,|Du|\, g_t(x,|\delta u|)\,\dd x
\le &
\int_{B_R}
 \eta^p g^{\beta}(x,|\delta u|)
\,\delta g_t(x,|\delta u|)|\delta u|\,\dd x\nonumber\\
&+
\int_{B_R}
 \eta^p g^{\beta}(x,|\delta u|)
\,\delta\, g_t(x,|Du|)|Du|\,\dd x .
\end{align}
Applying hypothesis \eqref{oldH2} and \eqref{(24)DallAglioMascolo} to the first and the second integral on the right hand side of \eqref{estimate_time_derivative_term}, respectively, and recalling that $\eta\leq 1$,
we get
\begin{align*}
\int_{B_R}
\eta^p g^{\beta}(x,|\delta u|)\,\delta\, |Du|\,g_t(x,|\delta u|)\,\dd x
&\le
C
\int_{B_R}
\bigl[1+g^{\sigma+\beta}(x,|\delta u|)\bigr]\,\dd x\\
&\quad+\frac{c}{(R-\rho)^p}
\int_{B_R}\gamma^p(x)
\bigl[1+g^{\nu p+\beta +\sigma}(x,|\delta u|)\bigr]\,\dd x\\
&\le
\frac{c}{(R-\rho)^p}\int_{B_R}
\gamma^p(x)\bigl[1+g^{\nu p+\beta +\sigma}(x,|\delta u|)\bigr]\,\dd x,
\end{align*}
where $c$ depends on depends on  $\nu$, $\beta$, $p$, $\sigma$, $\delta$.
Returning to \eqref{doubleSmile} and recalling that $\eta\leq 1$ and $|D\eta|\leq 1/(R-\rho)$, we obtain
\begin{align}\label{Dw_estimate}
\int_{B_\rho} |Dw|\,\dd x
&\le 
\frac{c}{(R-\rho)^p}
\int_{B_R}\gamma^p(x)
\bigl[1+g^{\nu p+\beta +\sigma}(x,|\delta u|)\bigr]\,\dd x.
\end{align}
By Sobolev’s Embedding Theorem, \eqref{Dw_estimate} becomes
\begin{equation*}
\left(
\int_{B_\rho} |w|^{1^*}\,\dd x
\right)^{\frac{1}{1^*}}
\le
\int_{B_\rho} |Dw|\,\dd x
\le
\frac{c}{(R-\rho)^p}
\int_{B_R}\gamma^p(x)
\bigl[1+g^{\nu p+\beta +\sigma}(x,|\delta u|)\bigr]\,\dd x,
\end{equation*}
where $1^*=n/(n-1)$.
Recalling the definition of $w$ \eqref{definition_w}, the previous inequality can be rewritten as
\begin{equation}\label{(26)DallAglioMascolo}
\left(
\int_{B_\rho} \bigl[1+g^{\beta+1}(x,|\delta u|)\bigr]^{1^*}\,\dd x
\right)^{\frac{1}{1^*}}
\leq \frac{c}{(R-\rho)^p}
\int_{B_R}\gamma^p(x)
\bigl[1+g^{\nu p+\beta +\sigma}(x,|\delta u|)\bigr]\,\dd x.
\end{equation}
Since
\[
\bigl[1+g^{\beta+1}(|\delta u|)\bigr]^{1^*}
\geq
c \bigl[1+g(x,|\delta u|)\bigr]^{(\beta+1) 1^*},
\]
we obtain
\begin{equation}\label{A1}
\left(
\int_{B_\rho}
\bigl[1+g(x,|\delta u|)\bigr]^{(\beta+1) 1^*}
\,\dd x
\right)^{\frac{1}{1^*}}
\leq
\frac{c}{(R-\rho)^p}
\int_{B_R}\gamma^p(x)
\bigl[1+g^{\nu p+\beta +\sigma}(x,|\delta u|)\bigr]\,\dd x.
\end{equation}
Let us focus on the integral in the right-hand side of the previous inequality. Using H\"older inequality with exponents $\frac{s}{p}$ and $\frac{s}{s-p}$, we obtain
\begin{align}\label{second termwithgamma}
    \int_{B_R}  \gamma^p(x) \bigl[1+g^{\nu p+\beta +\sigma}(x,|\delta u|)\bigr]\,\dd x &\leq \left(\int_{B_R} \gamma^s(x) \,\dd x\right)^{\frac{p}{s}} \left(\int_{B_R} \bigl[1+g^{\nu p+\beta +\sigma}(x,|\delta u|)\bigr]^{\frac{s}{s-p}} \,\dd x\right)^{\frac{s-p}{s}}\nonumber\\
    &\leq \left(\int_{\Omega_0} \gamma^s(x) \,\dd x\right)^{\frac{p}{s}}  \left(\int_{B_R} \bigl[1+g^{\nu p+\beta +\sigma}(x,|\delta u|)\bigr]^{\frac{s}{s-p}} \,\dd x\right)^{\frac{s-p}{s}}\nonumber\\
    &= c \left(\int_{B_R} \bigl[1+g^{\nu p+\beta +\sigma}(x,|\delta u|)\bigr]^{\frac{s}{s-p}} \,\dd x\right)^{\frac{s-p}{s}},
\end{align}
where in the last step we have used that $\gamma\in L^s_{\rm{loc}}(\Omega)$ and $c=\|\gamma\|^p_{L^s(\Omega_0)}$.\\

\noindent We observe that
\[
\bigl[1+g^{\nu p+\beta +\sigma}(x,|\delta u|)\bigr]
\leq \bigl[1+g^{\beta}(x,|\delta u|)\bigr]
\bigl[1+g^{\nu p+\sigma}(x,|\delta u|)\bigr].
\]
Let \begin{equation}\label{theta}
    \frac{s}{s-p}<\theta<1^*
\end{equation} to be choosen opportunely lately, again by H\"older's inequality with exponents $\frac{\theta(s-p)}{s}$ and $\frac{\theta(s-p)}{\theta(s-p)-s}$,
we have 
\begin{align*}
\int_{B_R} &\bigl[1+g^{\nu p+\beta +\sigma}(x,|\delta u|)\bigr]^{\frac{s}{s-p}} \,\dd x \\
&\leq \left(\int_{B_R}\bigl[1+g^{\beta}(x,|\delta u|)\bigr]^{\theta}\,\dd x\right)^{\frac{s}{\theta(s-p)}} \, \left(\int_{B_R}\bigl[1+g^{\nu p+\sigma}(x,|\delta u|)\bigr]^{\frac{s\theta}{\theta(s-p)-s}}\,\dd x\right)^{1-\frac{s}{\theta(s-p)}}.
\end{align*}
Since $\sigma$ is strictly less than $1^*$, there exists $0<\epsilon<1$ such that $\nu p + \sigma= 1^* \epsilon$. We now choose $\theta$ depending on $\epsilon$ as $$\theta=\frac{s}{(1-\epsilon)s-p}.$$
We point out that this choice is admissible since it satisfies \eqref{theta}. Indeed, on the one hand $$\theta=\frac{s}{(1-\epsilon)s-p} > \frac{s}{s-p}$$ trivially because $s-p>(s-p)-s\epsilon$. On the other hand, to get $$\theta=\frac{s}{(1-\epsilon)s-p} < 1^*,$$ that is $s<1^*(s-p)-1^*\epsilon s$, it's enough to choose $$\epsilon< \frac{1^*(s-p)-s}{1^* s}<1.$$
We remark that  $$\frac{1^*(s-p)-s}{1^* s}>0$$ because, by hypothesis, we are assuming $s>pn$ which implies that $1^*(s-p)-s>0$.
With this choice of $\theta$, since $(\nu p +\sigma)=1^* \epsilon$, that is $(\nu p +\sigma) \frac{1}{\epsilon}=1^*$, we have that \begin{equation}\label{(27)DallAglioMascolo}
\frac{(\nu p+\sigma)s\theta}{\theta(s-p)-s} =1^*,
\end{equation}
and we get
\begin{align*}
    \int_{B_R} & \gamma(x) \bigl[1+g^{\nu p+\beta +\sigma}(x,|\delta u|)\bigr]\,\dd x \\ &\leq c \left(\int_{B_R}\bigl[1+g^{\beta}(x,|\delta u|)\bigr]^{\theta}\,\dd x\right)^{\frac{1}{\theta}} \, \left(\int_{B_R}\bigl[1+g^{1^*}(x,|\delta u|)\bigr]\,\dd x\right)^{\frac{s-p}{s}-\frac{1}{\theta}},
\end{align*}
where $c$ depends on $s$, $\theta$ and $\|\gamma\|_{L^s(\Omega_0)}$.\\

\noindent By using the standard properties $1+t^{1^*}\leq (1+t)^{1^*}$ and $(1+t)^\theta\leq 2^{\theta-1}(1+t^\theta)$, inequality \eqref{(26)DallAglioMascolo}
becomes
\begin{align}\label{(28)DallAglioMascolo}
\Bigg(
\int_{B_\rho} &\bigl[1+g^{1^*b}(x,|\delta u|)\bigr]\,\dd x
\Bigg)^{\frac{1}{1^*}}\nonumber\\
&\leq \frac{c\,b }{(R-\rho)^p}
\left(\int_{B_R}\bigl[1+g^{\theta b}(x,|\delta u|)\bigr]\,\dd x\right)^{\frac{1}{\theta}} \, \left(\int_{B_R}\bigl[1+g^{1^*}(x,|\delta u|)\bigr]\,\dd x\right)^{\frac{s-p}{s}-\frac{1}{\theta}},
\end{align}
where $b:=\beta+1$, $c$ depends on $C$, $s$, $\theta$ and $\|\gamma\|_{L^s(\Omega_0)}$.
This holds by choosing opportunely $\delta<1$ such that $g(x,|\delta u|)\in L^{\theta b}(B_R)$.

\noindent The next step of the proof consists in an iteration procedure.
We start by fixing $\bar{R}$ and $\bar{\rho}$ such that $0<\bar{\rho}<\bar{R}\leq \min\{1, {\rm{dist}}(x_0,\partial \Omega_0)\}$. We define 
\begin{equation*}
    \rho_j=\bar{\rho}+ \frac{\bar{R}-\bar{\rho}}{2^{j-1}}, \qquad b_j=\left(\frac{1^*}{\theta}\right)^j, \qquad A_j=\left(\int_{B_{\rho_j}}\bigl[1+g^{\theta b_j}(x,|\delta u|)\bigr]\,\dd x\right)^{\frac{1}{\theta b_j}}, \quad \mbox{for every $j=1,2,...$.}
\end{equation*}
\noindent For $\rho=\rho_{j+1}$, $R=\rho_j$ and $b=b_j$, \eqref{(28)DallAglioMascolo} gives
\begin{equation}\label{(29)DallAglioMascolo} 
 A_{j+1}\leq  \left(\frac{cM}{\delta^{p-1}(\bar{R}-\bar{\rho})^p}\right)^{\frac{1}{b_j}}  b_j^{\frac{1}{b_j}} 2^{\frac{jp}{b_j}}A_j, \quad \mbox{with }  M=\left(\int_{\bar{R}}\bigl[1+g^{1^*}(x,|\delta u|)\bigr]\,\dd x\right)^{\frac{s-p}{s}-\frac{1}{\theta}}.
\end{equation}
By iterating \eqref{(29)DallAglioMascolo}, we obtain
\begin{equation*}
 A_{j+1}\leq  \left(\frac{cM}{\delta^{p-1}(\bar{R}-\bar{\rho})^p}\right)^{\sum\limits_{k=1}^j \frac{1}{b_k}}  \left(\prod_{k=1}^j b_k^{\frac{1}{b_k}}\right)  2^{p\sum\limits_{k=1}^j\frac{k}{b_k}}A_1.
\end{equation*}
We remark that Theorem \ref{Immersion Theorem} ensures that $M$, $A_1$ and therefore $A_j$, for every $j=1,2,...$, are finite. 
Since \begin{equation*}
   \sum\limits_{k=1}^\infty \frac{1}{b_k}=\frac{\theta}{1^*-\theta}, \quad \sum\limits_{k=1}^\infty \frac{k}{b_k}=\sum\limits_{k=1}^\infty k \left(\frac{\theta}{1^*}\right)^k, \quad \prod_{k=1}^\infty b_k^{\frac{1}{b_k}}= \rm{exp}\left(\sum\limits_{k=1}^\infty \frac{\ln{b_k}}{b_k}\right)<\infty,
\end{equation*}
we conclude that, for every $j=1,2,...$,
\begin{equation*}
    A_j \leq c \left(\frac{M}{\delta^{p-1}(\bar{R}-\bar{\rho})^p}\right)^{\frac{\theta}{1^*-\theta}}\left(\int_{\bar{R}}\bigl[1+g^{1^*}(x,|\delta u|)\bigr]\,\dd x\right)^{\frac{1}{1^*}},
\end{equation*}
where $c$ depends on $n$, $C$, $s$, $\theta$ and $\|\gamma\|_{L^s(\Omega_0)}$.
By definition of $M$ and noticing that $g$ is increasing and $\delta<1$, the previous inequality becomes
\begin{equation}\label{(30)DallAglioMascolo}
  A_j \leq c \left(\frac{1}{\delta^{\frac{(p-1)s(n-1)}{s-pn}}(\bar{R}-\bar{\rho})^{\frac{ps(n-1)}{s-pn}}}\left(\int_{\bar{R}}\bigl[1+g^{1^*}(x,| u|)\bigr]\,\dd x\right)^{\frac{1}{1^*}}\right)^\alpha, 
\end{equation}
where $\alpha=\alpha(\theta)=\frac{\theta(s-pn)}{s(n-\theta(n-1))}$ and $c$ depends on $\alpha$ (or, which is equivalent, on
$\theta$), $n$, $p$, $s$ and $\|\gamma\|_{L^s(\Omega_0)}$.
We note that, depending on the choice of $\theta$, the exponent $\alpha$ can assume any value greater than $1$.
Thus, we get
\begin{equation}\label{(31)_0DallAglioMascolo}
\sup\limits_{B_{\bar{\rho}}}g(x, |\delta u|)= \lim\limits_{j \to +\infty} \left(\int_{B_{\bar{\rho}}}g^{b_j \theta}(x, |u|)\, \dd x \right)^{\frac{1}{b_j \theta}} \leq \limsup\limits_{j \to +\infty} A_j.
\end{equation}
We now apply our hypothesis $\eqref{ipotesi1}$ on the right-hand side of \eqref{(31)_0DallAglioMascolo}, with $\lambda=\frac{1}{\delta}>1$ (since $\delta<1$), $t=\delta\vert u\vert$, to obtain

$$
g(x,\vert u\vert)=g\left(x,\frac{1}{\delta}(\delta\vert u\vert)\right)\leq c\delta^{-p}g^{\sigma}(x,\delta\vert u\vert),
$$ 

\noindent and so
\begin{equation}\label{(31)DallAglioMascolo}
\sup\limits_{B_{\bar{\rho}}}g(x, | u|)\leq c\delta^{-p}\left(\limsup\limits_{j \to +\infty} A_j\right)^\sigma.
\end{equation}
Replacing $\bar{R}$ and $\bar{\rho}$ with $R$ and $\rho$ respectively, the thesis \eqref{boundedness inequality} follows from \eqref{(30)DallAglioMascolo} and \eqref{(31)DallAglioMascolo}.
\end{proof}

\section*{Acknowledgement}
The authors wish to express their gratitude to Professor Paolo Marcellini for providing challenging food for thought for this research.

E. Mascolo and A. Nastasi are members of the \textit{Gruppo Nazionale per
l'Analisi Matematica, la Probabilit\`{a} e le loro Applicazioni} (GNAMPA) of
the Istituto Nazionale di Alta Matematica (INdAM).

\end{document}